\documentclass[12pt,reqno]{amsart}
\usepackage[margin=2.55cm]{geometry}
\usepackage[T1]{fontenc}
\usepackage{lmodern}
\usepackage{amsmath,amssymb,amsthm,mathtools,mathrsfs}
\usepackage{enumitem}
\usepackage{microtype}
\usepackage{xcolor}
\usepackage[numbers,sort&compress]{natbib}
\usepackage{hyperref}
\usepackage{aliascnt}
\usepackage[nameinlink,noabbrev]{cleveref}

\hypersetup{
  colorlinks=true,
  linkcolor=blue!55!black,
  citecolor=blue!55!black,
  urlcolor=blue!55!black,
  pdftitle={Weak-A-infinity packing and estimates for fractional sparse forms on spaces of homogeneous type},
  pdfauthor={Hanwen Zhang}
}

\numberwithin{equation}{section}
\setlist[enumerate]{leftmargin=2.2em,itemsep=0.2em,topsep=0.25em}
\setlist[itemize]{leftmargin=2.0em,itemsep=0.2em,topsep=0.25em}

\newtheorem{theorem}{Theorem}[section]
\newaliascnt{proposition}{theorem}
\newtheorem{proposition}[proposition]{Proposition}
\aliascntresetthe{proposition}
\newaliascnt{lemma}{theorem}
\newtheorem{lemma}[lemma]{Lemma}
\aliascntresetthe{lemma}
\newaliascnt{corollary}{theorem}
\newtheorem{corollary}[corollary]{Corollary}
\aliascntresetthe{corollary}

\theoremstyle{definition}
\newaliascnt{definition}{theorem}
\newtheorem{definition}[definition]{Definition}
\aliascntresetthe{definition}
\newaliascnt{convention}{theorem}

\aliascntresetthe{convention}

\newtheorem*{conventionplain}{Convention}
\theoremstyle{remark}
\newaliascnt{remark}{theorem}

\aliascntresetthe{remark}
\newtheorem*{remarkplain}{Remark}

\crefname{theorem}{Theorem}{Theorems}
\Crefname{theorem}{Theorem}{Theorems}
\crefname{proposition}{Proposition}{Propositions}
\Crefname{proposition}{Proposition}{Propositions}
\crefname{lemma}{Lemma}{Lemmas}
\Crefname{lemma}{Lemma}{Lemmas}
\crefname{corollary}{Corollary}{Corollaries}
\Crefname{corollary}{Corollary}{Corollaries}
\crefname{remark}{Remark}{Remarks}
\Crefname{remark}{Remark}{Remarks}
\crefname{definition}{Definition}{Definitions}
\Crefname{definition}{Definition}{Definitions}
\crefname{convention}{Convention}{Conventions}
\Crefname{convention}{Convention}{Conventions}

\newcommand{\D}{\mathscr D}
\newcommand{\Sfam}{\mathcal S}
\newcommand{\Ffam}{\mathcal F}
\newcommand{\one}{\mathbf 1}
\newcommand{\dd}{\,\mathrm d}
\newcommand{\norm}[2]{\left\lVert #1\right\rVert_{#2}}
\newcommand{\avg}[1]{\left\langle #1\right\rangle}
\newcommand{\weakA}{A_{\infty}^{\mathrm{weak}}}

\newcommand{\Achar}{\mathfrak A}
\newcommand{\Nform}{\mathfrak N}
\newcommand{\Lform}{\mathscr L}
\newcommand{\eps}{\varepsilon}
\newcommand{\sgap}{\Delta}
\newcommand{\wideB}[1]{\widehat B_{#1}}

\title[Weak-$A_\infty$ packing and estimates]
{Weak-\texorpdfstring{$A_\infty$}{A-infinity} packing and estimates for fractional sparse forms on spaces of homogeneous type}
\author[H. Zhang]{Hanwen Zhang}
\address{Hanwen Zhang\\
Zhejiang Ocean University\\
Zhoushan 316022\\
People's Republic of China}
\email{hanwzhang@aliyun.com}
\date{July 21, 2026}
\subjclass[2020]{42B25, 42B35}
\keywords{Fractional sparse forms, two-weight inequalities, weak $A_\infty$ weights, sparse domination, spaces of homogeneous type}

\begin{document}

\begin{abstract}
We prove two-weight estimates for fractional sparse forms on spaces of homogeneous type.  The relevant transformed weights are assumed to belong to weak-$A_\infty$.  The weak reverse H\"older inequality is normalized on enlarged balls, while sparse testing is normalized on dyadic cubes.  To measure this loss, we introduce
$\mathfrak d_Q^w:=w(Q)/w(2\Lambda B_Q)$.  A packing estimate for the defect levels, combined with two-weight testing, gives the required mixed weak-$A_\infty$ bounds.
\end{abstract}

\maketitle

\section{Introduction and main results}

Sparse domination reduces many weighted problems to positive dyadic forms.
This applies to Calder\'on--Zygmund and fractional operators, maximal
operators, and square-function models; see \citet{Lerner2016,Lorist}.
For two weights, the resulting forms are treated by Sawyer-type testing for
positive dyadic operators \citep{Sawyer1988,LSUT,COV,HHL,FH}.

Sharp weighted estimates for classical fractional integral operators were
established by \citet{LMPT} in the Euclidean setting and subsequently studied by \citet{Kairema2014} on spaces of homogeneous type. In a broader operator-theoretic framework, \citet{BZ} proved sparse domination and quantitative two-weight estimates for fractional powers
$L^{-\alpha/\kappa}$ of sectorial operators whose semigroups satisfy suitable off-diagonal estimates. Their framework includes both the classical Riesz potential and fractional powers of divergence-form elliptic operators, and reduces the corresponding estimates to bounds for a general bilinear fractional sparse form. Under classical $A_\infty$ assumptions, they obtained mixed bounds for this form. The present paper takes this sparse-form problem as its starting point and extends the corresponding estimates to weak-$A_\infty$ weights on spaces of homogeneous type.

The full definitions and basic properties of spaces of homogeneous type, weak-$A_\infty$ weights, dyadic systems, and sparse families will be recalled in Section~2. Within the introduction, we only introduce essential definitions to illustrate our main theorem, as well as the proof strategy and the motivation behind the definition of the enlargement defects.

Let $(X,\rho,\mu)$ be a space of homogeneous type: $\rho$ is a quasi-metric
with quasi-triangle constant $\Lambda\ge1$, and $\mu$ is a doubling Borel
measure.  Fix one of the dyadic systems $\D$ of \citet{HK}, and let
$\Sfam\subset\D$ be a $1/2$-sparse family.

\begin{conventionplain}
	All sparse families considered below are finite or countable, and sums over countable sparse families are understood as monotone sums of nonnegative terms.  In the proofs of the main estimates, we first work with finite sparse families; the passage to countable families is justified by the monotone truncation argument in \cref{lem:finite-truncations}.
\end{conventionplain}

For such a family, we study the bilinear fractional sparse form
\begin{equation*}
	\Lform_{\Sfam}^{a,b,\gamma}(f,g)
	:=\sum_{Q\in\Sfam}
	\avg{f}_{a,Q}\avg{g}_{b',Q}\mu(Q)^{1+\gamma},
\end{equation*}
where
\begin{equation*}
	\avg{h}_{t,Q}
	:=\left(\frac1{\mu(Q)}\int_Q|h|^t\,\dd\mu\right)^{1/t},
	\qquad 0<t<\infty,
\end{equation*}
with the usual essential-supremum interpretation when $t=\infty$.
For a pair of weights $(\omega,\sigma)$, define
\begin{equation*}
	\Nform_{\Sfam}^{a,b,\gamma;p,q}(\omega,\sigma)
	:=\sup_{0\ne f,g\ge0}
	\frac{\Lform_{\Sfam}^{a,b,\gamma}(f,g)}
	{\norm{f}{L^p(\omega)}\norm{g}{L^{q'}(\sigma)}},
\end{equation*}
and introduce the transformed weights
\begin{equation*}
	u:=\omega^{-\frac{a}{p-a}},
	\qquad
	v:=\sigma^{-\frac{b'}{q'-b'}}.
\end{equation*}
Our aim is to control this norm quantitatively under weak-$A_\infty$ assumptions on $u$ and $v$. The passage from classical $A_\infty$ to weak-$A_\infty$ is not a mechanical replacement of weight characteristics.  In the classical setting, the Fujii--Wilson condition and the disjoint sparse major subsets give the cube-local packing estimate
\[
\sum_{\substack{Q\in\Sfam\\Q\subseteq R}}w(Q)
\lesssim
\int_R M(w\one_R)\,\dd\mu
\lesssim
[w]_{A_\infty}w(R).
\]
For weak-$A_\infty$ on a space of homogeneous type, however, the
corresponding maximal and reverse H\"older estimates are normalized by
the weight of an enlargement of $R$.
By contrast, the Sawyer testing conditions for the sparse form are normalized by $w(R)$. 

A related use of
weak-$A_\infty$ estimates appears in the work of \citet{Li} on
Calder\'on--Zygmund operators. There, a quantitative power-bump estimate is
combined with the weak reverse H\"older inequality, which determines
admissible bump exponents and leads to bounds for the associated sparse
testing quantities. In the fractional setting considered here, we instead
work directly with the sparse form and encode the discrepancy between the
enlarged-cube and original-cube normalizations through enlargement defects.

We use the weak-$A_\infty$ class of \citet{AHT}, with dilation
$2\Lambda$:
\begin{equation*}
	[w]_{\weakA}
	:=\sup_B\frac1{w(2\Lambda B)}
	\int_B M(w\one_B)\,\dd\mu,
\end{equation*}
where balls with $w(2\Lambda B)=0$ are omitted.  Other fixed dilations larger
than $\Lambda$ give the same class, with quantitatively comparable
characteristics.

For every $Q\in\D$, choose an associated ball $B_Q$ from the dyadic
construction so that
\begin{equation*}
	Q\subseteq B_Q,
	\qquad
	\mu(B_Q)\simeq_X\mu(Q),
\end{equation*}
and put
\begin{equation*}
	\wideB{Q}:=2\Lambda B_Q.
\end{equation*}
Then $Q\subseteq\wideB{Q}$ and $\mu(\wideB{Q})\simeq_X\mu(Q)$.
In the present setting, the same sparse-major-subset argument gives only
\begin{align*}
	\sum_{\substack{Q\in\Sfam\\Q\subseteq R}}w(Q)
	&\lesssim_X \int_R M(w\one_R)\,\dd\mu \\
	&\le \int_{B_R}M(w\one_{B_R})\,\dd\mu
	\lesssim_X [w]_{\weakA}w(\wideB{R}).
\end{align*}
 Since weak-$A_\infty$ does not
provide a uniform comparison $w(\wideB{R})\lesssim w(R)$, the classical
cube-local packing argument does not close.

We define the ratio
\begin{equation*}
	\mathfrak d_Q^w:=
	\begin{cases}
		\dfrac{w(Q)}{w(\wideB{Q})},&w(\wideB{Q})>0,\\[0.8em]
		0,&w(\wideB{Q})=0.
	\end{cases}
\end{equation*}
and call it the \emph{enlargement defect}.  Individual defects can be
arbitrarily small.  The key geometric result proved in Section~2 is instead
the distributional estimate
\begin{equation}\label{eq:intro-defect-packing}
	\sum_{\substack{Q\in\Sfam,\ Q\subseteq R\\
			\mathfrak d_Q^w\ge t}}
	w(Q)
	\lesssim_X
	[w]_{\weakA}\bigl(1+\log_2(1/t)\bigr)w(R),
	\qquad 0<t\le1.
\end{equation}
Thus the defect levels have quantitative Carleson control even though no
cube-by-cube lower bound is available.

With this packing lemma for weak-$A_\infty$ weights, we obtain the following main theorem.

Let $\Sfam\subset\D$ be a finite or countable $1/2$-sparse family.  Assume
that $\omega$ and $\sigma$ are locally integrable and positive almost
everywhere.  Define
\begin{equation*}
	u:=\omega^{-\frac{a}{p-a}},
	\qquad
	v:=\sigma^{-\frac{b'}{q'-b'}}.
\end{equation*}
and assume that $u$ and $v$ are locally integrable.  Suppose
\begin{equation*}
	1\le p\le q<b\le\infty,
	\qquad q>1,
	\qquad 0<a<p,
	\qquad
	\frac1p-\frac1q\le\gamma\le\frac1a-\frac1b.
\end{equation*}
We use $1/\infty=0$ and $\infty'=1$, and write
\begin{equation*}
	\sgap:=\frac1p-\frac1q,
	\qquad
	\theta_u:=\frac1a-\frac1p,
	\qquad
	\theta_v:=\frac1q-\frac1b,
	\qquad
	\kappa:=\frac1a-\frac1b-\gamma.
\end{equation*}
For the outer balls above, define
\begin{align*}
	\Achar_{\Sfam}^{\gamma}(u,v)
	&:={}
	\sup_{Q\in\Sfam}
	\mu(\wideB{Q})^{-\kappa}
	u(\wideB{Q})^{\theta_u}v(\wideB{Q})^{\theta_v},\\
	\Achar^{\gamma}(u,v)
	&:={}
	\sup_B
	\mu(B)^{-\kappa}u(B)^{\theta_u}v(B)^{\theta_v},
\end{align*}
where the second supremum is over balls of positive finite measure.  Clearly,
\begin{equation*}
	\Achar_{\Sfam}^{\gamma}(u,v)
	\le\Achar^{\gamma}(u,v).
\end{equation*}

For convenience, write
\begin{equation*}
	\mathfrak m_{p,q}(U,V)
	:=\min\left\{
	U^{1/p}V^{1/p'},
	U^{1/q}V^{1/q'}
	\right\},
\end{equation*}
with the usual convention when $p=1$.  Whenever the corresponding
weak-$A_\infty$ assumption is in force, we abbreviate
$U:=[u]_{\weakA}$ and $V:=[v]_{\weakA}$.

We distinguish the following two mutually exclusive regimes:
\begin{align*}
	\text{\rm Case 1:}\qquad
	&\gamma<\frac1a-\frac1b
	\quad\text{and}\quad
	\bigl(p<q\ \text{or}\ \gamma=0\bigr),\\
	\text{\rm Case 2:}\qquad
	&\gamma=\frac1a-\frac1b
	\quad\text{or}\quad
	\bigl(p=q\ \text{and}\ \gamma>0\bigr).
\end{align*}

\begin{theorem}[Weak-$A_\infty$ estimate]
	\label{thm:weakA-bilinear-intro}
	Let the parameters, sparse family, weights, and transformed weights satisfy
	the hypotheses above. Assume that $u\in\weakA$. In Case~1, assume additionally
	that $v\in\weakA$ when $p>1$; in Case~2, assume that $v\in\weakA$. Then
	\begin{equation*}
		\Nform_{\Sfam}^{a,b,\gamma;p,q}(\omega,\sigma)
		\lesssim
		\Achar_{\Sfam}^{\gamma}(u,v)
		\begin{dcases}
			[u]_{\weakA}^{1/q}
			+\one_{\{p>1\}}[v]_{\weakA}^{1/p'},
			& \text{\rm Case 1},\\[1em]
			[u]_{\weakA}^{1/q}\mathfrak m_{p,q}(U,V)^{1/q'}
			+\one_{\{p>1\}}
			[v]_{\weakA}^{1/p'}\mathfrak m_{p,q}(U,V)^{1/p},
			& \text{\rm Case 2}.
		\end{dcases}
	\end{equation*}
	The implicit constant depends only on the exponents and the structural
	constants of the space. In both cases,
	$\Achar_{\Sfam}^{\gamma}(u,v)$ may be replaced by
	$\Achar^{\gamma}(u,v)$ whenever the latter is finite.
\end{theorem}

\begin{remarkplain}
	At $\gamma=1/a-1/b$, the normalizing measure power vanishes and
	\[
	\Achar^{\gamma}(u,v)
	=\sup_Bu(B)^{\theta_u}v(B)^{\theta_v}.
	\]
	On an unbounded space this may be infinite even for elementary weights.
	The family-local estimate still applies to fixed or truncated sparse
	families; global applications require finiteness of the critical
	characteristic.
\end{remarkplain}

The remainder of the paper is organized as follows.  Section~2 develops the
weak-$A_\infty$ geometry of the dyadic outer balls and proves the defect
packing estimate \eqref{eq:intro-defect-packing}.  Section~3 combines this
packing result with Sawyer-type testing to prove
\cref{thm:weakA-bilinear-intro}, treating Case~1 and Case~2
separately.  In Section~4 we apply the bilinear estimates proved above to the
pointwise $\ell^r$-sparse operators studied by \citet{FH}.

\section{\texorpdfstring{Weak $A_\infty$ geometry and packing estimates}
{Weak A-infinity geometry and packing estimates}}

The framework of spaces of homogeneous type, introduced by Coifman and
Weiss \citep{CoifmanWeiss}, allows much of real-variable harmonic analysis
to be developed in the absence of a linear or Euclidean structure.

\begin{definition}[Space of homogeneous type]
A space of homogeneous type is a triple $(X,\rho,\mu)$, where $\rho$ is a
quasi-metric and $\mu$ is a doubling Borel measure:
\begin{enumerate}
 \item $\rho(x,y)=0$ if and only if $x=y$;
 \item $\rho(x,y)=\rho(y,x)$ for all $x,y\in X$;
 \item there is a constant $\Lambda\ge1$ such that
 \[
  \rho(x,z)\le \Lambda\bigl(\rho(x,y)+\rho(y,z)\bigr)
 \]
 for all $x,y,z\in X$;
 \item there is a constant $C_\mu\ge1$ such that
 \[
  \mu(B(x,2r))\le C_\mu\mu(B(x,r))
 \]
 for all $x\in X$ and $r>0$.
\end{enumerate}
Here $B(x,r):=\{y\in X:\rho(x,y)<r\}$, and $\lambda B(x,r):=B(x,\lambda r)$.  We assume that balls are Borel sets and that $0<\mu(B)<\infty$ for every ball.  The doubling property implies geometric doubling: every ball of radius $r$ can be covered by at most $N_X=N(\Lambda,C_\mu)$ balls of radius $r/2$.
\end{definition}

A \emph{weight} is a nonnegative locally integrable function.  Let $M$ be
the non-centered Hardy--Littlewood maximal operator over balls.  For a fixed
dilation parameter $\lambda\ge1$, set
\begin{equation*}
 [w]_{A_\infty^\lambda}
 :=\sup_B\frac1{w(\lambda B)}
 \int_BM(w\one_B)\,\dd\mu,
\end{equation*}
where the supremum is over balls $B$ with $w(\lambda B)>0$; balls with $w(\lambda B)=0$ make no contribution.

We shall use the following facts from \citep[Theorem~3.3, Proposition~3.9,
and Corollary~5.5]{AHT}.

\begin{proposition}\label{prop:AHT-properties}
The following assertions hold.
\begin{enumerate}
 \item If $\lambda,\lambda'>\Lambda$, then
 \[
  A_\infty^\lambda=A_\infty^{\lambda'}.
 \]
 Moreover, their characteristics are quantitatively comparable, with comparison constants depending only on $\Lambda,C_\mu,\lambda$, and $\lambda'$.  We therefore fix
 \[
  \weakA:=A_\infty^{2\Lambda},
  \qquad
  [w]_{\weakA}:=[w]_{A_\infty^{2\Lambda}}.
 \]
 \item Every nonzero $w\in\weakA$ satisfies
 \begin{equation}\label{eq:lowA}
  [w]_{\weakA}\ge c_{0,X}
  :=\frac1{2(2\Lambda)^{\log_2N_X}}.
 \end{equation}
 \item There are structural constants $\tau_X,C_{\mathrm{RH},X}>0$ such that, whenever $w\in\weakA$ and
 \[
  0<\eps\le\frac1{\tau_X[w]_{\weakA}},
 \]
 one has
 \begin{equation}\label{eq:rhi}
  \left(\frac1{\mu(B)}\int_Bw^{1+\eps}\,\dd\mu\right)^{1/(1+\eps)}
  \le \frac{C_{\mathrm{RH},X}}{\mu(2\Lambda B)}\int_{2\Lambda B}w\,\dd\mu
 \end{equation}
 for every ball $B$.
\end{enumerate}
\end{proposition}

The reverse H\"older inequality in \eqref{eq:rhi} is formulated for balls, whereas the sparse forms and their testing conditions are indexed by dyadic cubes. To pass between these two geometries, we use the adjacent dyadic systems of Hyt\"onen and Kairema \citep[Theorems~2.2 and~4.1]{HK}.

\begin{theorem}\label{thm:HK-dyadic-systems}
Let $(X,\rho,\mu)$ be a space of homogeneous type.  There exist a scale parameter $0<\eta<1$, an integer $K\ge1$, and structural constants $0<c_{\D}\le C_{\D}<\infty$ and $C_{\mathrm{adj}}\ge1$ for which there are dyadic systems
\[
 \D^t=\bigcup_{k\in\mathbb Z}\D_k^t,
 \qquad t=1,\ldots,K,
\]
with each $\D_k^t$ countable, such that the following assertions hold.
\begin{enumerate}
 \item For every $t\in\{1,\ldots,K\}$ and $k\in\mathbb Z$, the cubes in $\D_k^t$ form a disjoint partition of $X$:
 \[
  X=\mathop{\dot\bigcup}_{Q\in\D_k^t}Q.
 \]
 \item If $Q\in\D_k^t$ and $P\in\D_\ell^t$ with $\ell\ge k$, then either $P\subseteq Q$ or $P\cap Q=\varnothing$.
 \item Every $Q\in\D_k^t$ has a center $z_Q\in X$ such that
 \begin{equation}\label{eq:cube-ball}
  B(z_Q,c_{\D}\eta^k)\subseteq Q
  \subseteq B(z_Q,C_{\D}\eta^k).
 \end{equation}
 \item For every ball $B=B(x,r)$ there exist $t\in\{1,\ldots,K\}$ and $Q\in\D^t$ such that
 \begin{equation*}
  B\subseteq Q,
  \qquad
  \operatorname{diam}_{\rho}(Q)\le C_{\mathrm{adj}}r.
 \end{equation*}
\end{enumerate}
The number $K$ and all the displayed constants depend only on the structural constants of the space.
\end{theorem}

The adjacent systems are useful when balls must be covered by comparable dyadic cubes. For the defect-packing argument itself, we fix one system and use only its laminar structure and the comparison between each cube and its associated outer ball. Write 
\[
 \D=\bigcup_{k\in\mathbb Z}\D_k.
\]
The family $\D$ is laminar, and each cube has one ancestor at every coarser
generation.  For $Q\in\D_k$, put
\begin{equation*}
 B_Q:=B(z_Q,C_{\D}\eta^k)
 \qquad\text{and}\qquad
 \wideB{Q}:=2\Lambda B_Q.
\end{equation*}
Then
\begin{equation*}
 Q\subseteq B_Q\subseteq\wideB{Q},
 \qquad
 \mu(\wideB{Q})\simeq_X\mu(Q).
\end{equation*}
The comparison follows from \eqref{eq:cube-ball} and doubling.  
The ball $\wideB{Q}$ is the enlargement in \eqref{eq:rhi}; the balls
$B_Q$ and $\wideB{Q}$ need not be nested as $Q$ varies.

\begin{definition}[Sparsity]
Let $0<\tau\le1$.  A finite or countable family $\Sfam\subset\D$ of distinct dyadic cubes is $\tau$-sparse if
\begin{equation*}
 \mu\left(
 Q\setminus\bigcup_{P\in\operatorname{ch}_{\Sfam}(Q)}P
 \right)\ge\tau\mu(Q)
\end{equation*}
for every $Q\in\Sfam$, where $\operatorname{ch}_{\Sfam}(Q)$ denotes the maximal strict members of $\Sfam$ contained in $Q$.  Equivalently,
\[
 \sum_{P\in\operatorname{ch}_{\Sfam}(Q)}\mu(P)
 \le(1-\tau)\mu(Q).
\]
Throughout the paper, ``sparse family'' means a family of distinct cubes satisfying this child-packing condition.
\end{definition}

The child-packing condition has two consequences that will be used
repeatedly.  It produces disjoint major subsets of the sparse cubes and, when iterated down the sparse tree, gives geometric decay in the underlying measure.

\begin{lemma}\label{lem:sparse-remainders}
Let $\Sfam\subset\D$ be a finite or countable $1/2$-sparse family, and define
\[
 F_Q:=Q\setminus
 \bigcup_{P\in\operatorname{ch}_{\Sfam}(Q)}P,
 \qquad Q\in\Sfam.
\]
Then the sets $\{F_Q:Q\in\Sfam\}$ are pairwise disjoint and
\[
 \mu(F_Q)\ge\frac12\mu(Q).
\]
Moreover, if $\operatorname{gen}_m(Q)$ denotes the $m$-th sparse generation below $Q$, then
\begin{equation*}
 \mu\left(\bigcup_{P\in\operatorname{gen}_m(Q)}P\right)
 \le2^{-m}\mu(Q),
 \qquad m\ge0.
\end{equation*}
\end{lemma}

\begin{proof}
If two sparse cubes are disjoint, then their remainders are disjoint.  If $P\subsetneq Q$, then $P$ is contained in a maximal proper sparse descendant of $Q$, and hence $P\cap F_Q=\varnothing$.  The lower bound for $\mu(F_Q)$ is the sparsity condition.  The generation estimate follows by induction from the child-packing inequality.
\end{proof}

The reverse H\"older inequality on $B_Q$ gives the following cube estimate.

\begin{corollary}\label{cor:weakA-set-decay}
Let $w\in\weakA$ be nonzero, and set
\[
 \beta_w:=\frac1{\tau_X[w]_{\weakA}+1}.
\]
Then
\begin{equation}\label{eq:beta}
 \beta_w\gtrsim_X\frac1{[w]_{\weakA}}.
\end{equation}
Moreover, for every $Q\in\D$ and every measurable $E\subset Q$,
\begin{equation}\label{eq:setdec}
 w(E)
 \lesssim_X w(\wideB{Q})
 \left(\frac{\mu(E)}{\mu(Q)}\right)^{\beta_w}.
\end{equation}
\end{corollary}

\begin{proof}
	Set
	\[
	\eps:=\frac{1}{\tau_X[w]_{\weakA}}.
	\]
	Then \eqref{eq:rhi} applies, and
	\[
	\frac{\eps}{1+\eps}
	=\frac{1}{\tau_X[w]_{\weakA}+1}
	=\beta_w.
	\]
	Let $Q\in\D$ and let $E\subset Q$ be measurable. Since
	$E\subset Q\subset B_Q$, H\"older's inequality gives
	\begin{align*}
		w(E)
		&=\int_E w\,\dd\mu\\
		&\le
		\left(\int_{B_Q}w^{1+\eps}\,\dd\mu\right)^{1/(1+\eps)}
		\mu(E)^{\eps/(1+\eps)}.
	\end{align*}
	Applying the weak reverse H\"older inequality \eqref{eq:rhi} to
	$B_Q$, whose enlargement is $\wideB{Q}=2\Lambda B_Q$, we obtain
	\begin{align*}
		w(E)
		&\le
		C_{\mathrm{RH},X}
		\frac{\mu(B_Q)^{1/(1+\eps)}}{\mu(\wideB{Q})}
		w(\wideB{Q})\,
		\mu(E)^{\eps/(1+\eps)}\\
		&=
		C_{\mathrm{RH},X}
		\frac{\mu(B_Q)}{\mu(\wideB{Q})}
		w(\wideB{Q})
		\left(\frac{\mu(E)}{\mu(B_Q)}\right)^{\beta_w}.
	\end{align*}
	Since $B_Q\subset\wideB{Q}$ and $Q\subset B_Q$, it follows that
	\[
	w(E)
	\le
	C_{\mathrm{RH},X}w(\wideB{Q})
	\left(\frac{\mu(E)}{\mu(Q)}\right)^{\beta_w},
	\]
	which proves \eqref{eq:setdec}.
	
	Finally, by \eqref{eq:lowA},
	$[w]_{\weakA}\ge c_{0,X}$, and hence
	\begin{align*}
		\beta_w
		&=\frac{1}{\tau_X[w]_{\weakA}+1}\\
		&=\frac{1}{[w]_{\weakA}}\,
		\frac{1}{\tau_X+[w]_{\weakA}^{-1}}\\
		&\ge
		\frac{1}{\tau_X+c_{0,X}^{-1}}\,
		\frac{1}{[w]_{\weakA}}.
	\end{align*}
	This proves \eqref{eq:beta} and completes the proof.
\end{proof}

For $Q\in\Sfam$, set
\[
 H_Q^{(m)}:=\bigcup_{P\in\operatorname{gen}_m(Q)}P.
\]
By \cref{lem:sparse-remainders,cor:weakA-set-decay}, every nonzero $w\in\weakA$ satisfies
\begin{equation}\label{eq:gdec}
 w(H_Q^{(m)})
 \lesssim_X 2^{-m\beta_w}w(\wideB{Q}).
\end{equation}
Since testing is normalized by $w(Q)$, we retain the ratio
$w(Q)/w(\wideB{Q})$.

\begin{definition}[Enlargement defect]
Let $w$ be a positive Borel measure that is finite on balls.  For $Q\in\D$, define
\begin{equation*}
 \mathfrak d_Q^w:=
 \begin{cases}
 \dfrac{w(Q)}{w(\wideB{Q})},&w(\wideB{Q})>0,\\[0.8em]
 0,&w(\wideB{Q})=0.
 \end{cases}
\end{equation*}
Then $\mathfrak d_Q^w\in[0,1]$.  Because $Q\subset\wideB{Q}$, a zero defect occurs exactly when $w(Q)=0$.
\end{definition}

For every $Q\in\Sfam$ with $\mathfrak d_Q^w>0$, \eqref{eq:gdec} becomes
\[
 \frac{w(H_Q^{(m)})}{w(Q)}
 \lesssim_X
 \frac{2^{-m\beta_w}}{\mathfrak d_Q^w}.
\]

We now fix a defect threshold $0<t \le 1$ and restrict attention to cubes satisfying $\mathfrak d_Q^w\ge t$.  On this subfamily, the preceding estimate yields uniform weighted decay after a controlled number of sparse generations.

\begin{lemma}\label{lem:N-step-weighted-decay}
Let $w\in\weakA$ be nonzero, let $0<t\le1$, and let
$\Ffam\subset\Sfam$ be a sparse subfamily satisfying
$\mathfrak d_Q^w\ge t$ for every $Q\in\Ffam$.  For $Q\in\Ffam$, let
\[
 H_{Q,\Ffam}^{(m)}
 :=\bigcup_{P\in\operatorname{gen}_m^{\Ffam}(Q)}P
\]
be the union of the $m$-th generation below $Q$ in the $\Ffam$-tree.
Fix a structural constant $C_X\ge1$ for which
\eqref{eq:setdec} holds, and set
\begin{equation}\label{eq:Nchoice}
 N=N(t,w)
 :=\left\lceil
 \frac{\log_2(2C_X/t)}{\beta_w}
 \right\rceil.
\end{equation}
Then
\begin{equation*}
 w\bigl(H_{Q,\Ffam}^{(N)}\bigr)
 \le\frac12w(Q),
 \qquad Q\in\Ffam.
\end{equation*}
\end{lemma}

\begin{proof}
Since a subfamily of a $1/2$-sparse family is again $1/2$-sparse,
\cref{lem:sparse-remainders} gives
\[
 \mu\bigl(H_{Q,\Ffam}^{(N)}\bigr)
 \le2^{-N}\mu(Q).
\]
Using \eqref{eq:setdec} and $\mathfrak d_Q^w\ge t$,
\begin{align*}
 w\bigl(H_{Q,\Ffam}^{(N)}\bigr)
 &\le C_X2^{-N\beta_w}w(\wideB{Q})\\
 &\le C_Xt^{-1}2^{-N\beta_w}w(Q)
 \le\frac12w(Q),
\end{align*}
by the choice of $N$ in \eqref{eq:Nchoice}.
\end{proof}

To convert this delayed decay into a Carleson estimate, we split the sparse tree into depth classes modulo $N(t,w)$.  Within each class, successive cubes are separated by enough generations for the weighted remainders to carry a fixed proportion of their mass.  This yields the following defect-packing estimate.

\begin{proposition}\label{prop:weakA-defect-packing}
Let $\Sfam\subset\D$ be a finite or countable $1/2$-sparse family and let $w\in\weakA$ be nonzero.  For $0<t\le1$ and $R\in\D$, set
\[
 \Ffam_t(R):=
 \{Q\in\Sfam:Q\subseteq R,\ \mathfrak d_Q^w\ge t\}.
\]
Then
\begin{equation*}
 \sum_{Q\in\Ffam_t(R)}w(Q)
 \lesssim_X
 [w]_{\weakA}\bigl(1+\log_2(1/t)\bigr)w(R).
\end{equation*}
The estimate is uniform in the sparse family.
\end{proposition}

\begin{proof}

It is enough to consider finite $\Ffam_t(R)$; the countable case follows by
monotone convergence.  Regard this family as a forest.  Maximal cubes have
depth zero, and depth increases by one along each edge.  For
$j=0,\ldots,N-1$, define the residue class
\[
 \Ffam_t^{(j)}(R)
 :=\left\{
 Q\in\Ffam_t(R):
 \operatorname{depth}_{\Ffam_t(R)}(Q)\equiv j\pmod N
 \right\}.
\]
Thus
\[
 \Ffam_t(R)=\bigcup_{j=0}^{N-1}\Ffam_t^{(j)}(R).
\]

Fix $j$ and $Q\in\Ffam_t^{(j)}(R)$.  Let
\[
 G_Q:=Q\setminus
 \bigcup_{P\in\operatorname{ch}_{\Ffam_t^{(j)}(R)}(Q)}P.
\]
A proper descendant of $Q$ in the same residue class lies $N$ generations
below it; terminated branches contribute nothing.  Hence
\[
 Q\setminus G_Q
 \subseteq H_{Q,\Ffam_t(R)}^{(N)}.
\]
By \cref{lem:N-step-weighted-decay},
\[
 w(Q\setminus G_Q)\le\frac12w(Q),
 \qquad
 w(G_Q)\ge\frac12w(Q).
\]

For fixed $j$, the sets $G_Q$ are pairwise disjoint.  In the nested case
the smaller cube lies in a maximal same-class descendant of the larger one.
All these sets are contained in $R$, so
\[
 \sum_{Q\in\Ffam_t^{(j)}(R)}w(Q)
 \le2\sum_{Q\in\Ffam_t^{(j)}(R)}w(G_Q)
 \le2w(R).
\]
Summing over the $N$ residue classes gives
\begin{equation}\label{eq:packN}
 \sum_{Q\in\Ffam_t(R)}w(Q)
 \le2N(t,w)w(R).
\end{equation}

\eqref{eq:Nchoice} and
\eqref{eq:beta} give
\[
 N(t,w)
 \le1+\frac{\log_2(2C_X/t)}{\beta_w}
 \lesssim_X
 [w]_{\weakA}\bigl(1+\log_2(1/t)\bigr).
\]
Together with \eqref{eq:packN}, this proves the proposition.
\end{proof}

The level-set estimate above is the natural geometric form of the packing
result.  In the testing arguments, however, the enlargement defects arise
through positive powers rather than through sharp level restrictions.
Indeed, the sparse coefficients will contain factors of the form
$(\mathfrak d_Q^w)^\eta$, where the exponent $\eta>0$ is determined by
the parameters of the sparse form.  We therefore record the integrated
consequence of the preceding proposition that will be used in
Section~3.

For $\eta>0$, define
\begin{equation}\label{eq:defect-moment-profile}
	\mathfrak C_{w,\eta}(\Sfam)
	:=\sup_{\substack{R\in\D\\w(R)>0}}
	\frac1{w(R)}
	\sum_{\substack{Q\in\Sfam\\Q\subseteq R}}
	(\mathfrak d_Q^w)^\eta w(Q).
\end{equation}

\begin{corollary}\label{cor:defect-moment-packing}
	Let $\Sfam\subset\D$ be a finite or countable $1/2$-sparse family, let
	$w\in\weakA$ be nonzero, and let $\eta>0$.  Then
	\begin{equation}\label{eq:defect-moment-bound}
		\mathfrak C_{w,\eta}(\Sfam)
		\lesssim_X
		[w]_{\weakA}\left(1+\frac1\eta\right).
	\end{equation}
\end{corollary}

\begin{proof}
	For $0\le d\le1$,
	\[
	d^\eta
	=\eta\int_0^1t^{\eta-1}\one_{\{d\ge t\}}\,\dd t.
	\]
	Tonelli's theorem and \cref{prop:weakA-defect-packing} give
	\begin{align*}
		&\sum_{\substack{Q\in\Sfam\\Q\subseteq R}}
		(\mathfrak d_Q^w)^\eta w(Q)\\
		&\quad\lesssim_X
		[w]_{\weakA}w(R)
		\eta\int_0^1t^{\eta-1}
		\bigl(1+\log_2(1/t)\bigr)\,\dd t\\
		&\quad\lesssim_X
		[w]_{\weakA}\left(1+\frac1\eta\right)w(R).
	\end{align*}
	Taking the supremum over $R$ proves the result.
\end{proof}

This completes the geometric part of the argument. The defect-moment estimate replaces the unavailable pointwise comparison $w(\widehat B_Q)\lesssim w(Q)$ by an averaged Carleson control. In the next section, we insert this control into the Sawyer testing estimates.

\section{Proofs of the main estimates}

In the proofs below, we may initially assume that the sparse family is finite. The countable case follows by the truncation argument in \cref{lem:finite-truncations}.

\begin{lemma}[Finite truncation]\label{lem:finite-truncations}
	Let $\Sfam\subset\D$ be a finite or countable $\tau$-sparse family, and let
	$\Sfam_N\uparrow\Sfam$ be increasing finite subfamilies.  Then every
	$\Sfam_N$ is $\tau$-sparse and
	\begin{align*}
		\Achar_{\Sfam_N}^{\gamma}(u,v)
		&\uparrow\Achar_{\Sfam}^{\gamma}(u,v),\\
		\Lform_{\Sfam_N}^{a,b,\gamma}(f,g)
		&\uparrow\Lform_{\Sfam}^{a,b,\gamma}(f,g).
	\end{align*}
	Consequently, estimates uniform over finite subfamilies pass to $\Sfam$.
\end{lemma}

\begin{proof}
	Every child of $Q$ in $\Sfam_N$ is contained in a child of $Q$ in $\Sfam$,
	so the child-packing inequality is preserved.  The two monotonicity
	statements are immediate, and the conclusion follows from monotone
	convergence.
\end{proof}

For $R\in\D$, define
\begin{equation*}
 T_Rh
 :=\sum_{\substack{Q\in\Sfam\\Q\subseteq R}}
 \avg{u}_{Q}^{1/a-1}
 \avg{v}_{Q}^{-1/b}
 \mu(Q)^\gamma
 \avg{h}_{Q}\one_Q,
\end{equation*}
where the factor involving $v$ is interpreted as $1$ when $b=\infty$.
Define
\[
 \zeta
 :=\sup_{\substack{R\in\D\\u(R)>0}}
 \frac{\norm{T_Ru}{L^q(v)}}{u(R)^{1/p}},
\]
and, when $p>1$,
\[
 \zeta^*
 :=\sup_{\substack{R\in\D\\v(R)>0}}
 \frac{\norm{T_Rv}{L^{p'}(u)}}{v(R)^{1/q'}}.
\]

The sparse norm is controlled by these two tests.

\begin{lemma}\label{lem:bilinear-testing}
Assume
\[
 1\le p\le q<b\le\infty,
 \qquad q>1,
 \qquad 0<a<p,
 \qquad
 \frac1p-\frac1q\le\gamma\le\frac1a-\frac1b.
\]
Then
\begin{equation*}
 \Nform_{\Sfam}^{a,b,\gamma;p,q}(\omega,\sigma)
 \lesssim
 \zeta+\one_{\{p>1\}}\zeta^*.
\end{equation*}
The implicit constant depends only on the exponents.
\end{lemma}

\begin{proof}
This is the usual Sawyer testing argument for positive dyadic operators;
see \citep[Lemmas~3.5--3.6]{BZ} and also \citep{FH,HHL,LSUT}.  It uses only
the laminar order, stopping cubes, and weighted dyadic maximal estimates, so
it applies to our fixed dyadic system.  For $p=1$ only the forward test is
needed.
\end{proof}

We also use the following finite-tree estimate.

\begin{lemma}\label{lem:laminar-operator}
Let $\mathcal G$ be a finite laminar family, let $\nu$ be a locally finite
positive measure, let $1<t<\infty$, and let $c_Q\ge0$.  Then
\begin{equation*}
 \left\|\sum_{Q\in\mathcal G}c_Q\one_Q\right\|_{L^t(\nu)}^t
 \lesssim_t
 \sum_{\substack{Q\in\mathcal G\\\nu(Q)>0}}
 c_Q\nu(Q)
 \left(
  \frac1{\nu(Q)}
  \sum_{\substack{P\in\mathcal G\\P\subseteq Q}}c_P\nu(P)
 \right)^{t-1}.
\end{equation*}
\end{lemma}

\begin{proof}
This is the upper estimate in the discrete tree-potential identity; see
\citep[Proposition~2.2]{COV} or \citep[Lemma~4.1]{FH}.  The proof only uses
laminarity and local finiteness.
\end{proof}

For $Q\in\Sfam$, put
\begin{equation*}
	\lambda_Q
	:=\mu(Q)^{-\kappa}u(Q)^{1/a}v(Q)^{1/b'}
	=K_Q u(Q)^{1/p}v(Q)^{1/q'},
\end{equation*}
where
\[
K_Q:=\mu(Q)^{-\kappa}u(Q)^{\theta_u}v(Q)^{\theta_v}.
\]
By the definition of enlargement defects,
\begin{equation*}
	K_Q
	\simeq_X
	\mu(\wideB{Q})^{-\kappa}
	u(\wideB{Q})^{\theta_u}v(\wideB{Q})^{\theta_v}
	(\mathfrak d_Q^u)^{\theta_u}
	(\mathfrak d_Q^v)^{\theta_v}.
\end{equation*}
In particular, if
$A:=\Achar_{\Sfam}^{\gamma}(u,v)$, then
\begin{equation}\label{eq:defect-factor-upper}
	K_Q
	\lesssim_X
	A(\mathfrak d_Q^u)^{\theta_u}
	(\mathfrak d_Q^v)^{\theta_v}
	\le A.
\end{equation}

To simplify notation, let $U:=[u]_{\weakA}$ and, whenever $v\in\weakA$, let $V:=[v]_{\weakA}$.

\subsection{Case~1}

We first record an elementary sparse summation estimate.

\begin{lemma}\label{lem:sparse-monomial}
	Let $\Ffam$ be a $1/2$-sparse family, let $R\in\D$, and let
	$A_0>0$ and $B,C\ge0$ satisfy $A_0+B+C\ge1$.  Then
	\[
	\sum_{\substack{Q\in\Ffam\\Q\subseteq R}}
	\mu(Q)^{A_0}u(Q)^Bv(Q)^C
	\lesssim_{A_0,B,C}
	\mu(R)^{A_0}u(R)^Bv(R)^C.
	\]
\end{lemma}

\begin{proof}
	First suppose $A_0+B+C=1$, and put $s:=B+C<1$.  If $s=0$, the assertion
	follows from the disjoint sparse remainders.  Otherwise, the localized
	dyadic maximal operator and the sparse major subsets give
	\[
	\sum_{Q\subseteq R}\mu(Q)^{A_0}u(Q)^Bv(Q)^C
	\lesssim
	\int_RM_R(u\one_R)^B M_R(v\one_R)^C\,\dd\mu.
	\]
	H\"older's inequality with total exponent $s$, followed by the localized
	Kolmogorov estimate, bounds the right-hand side by
	$\mu(R)^{1-s}u(R)^Bv(R)^C$.  The general case follows by normalizing the
	three exponents by $A_0+B+C$ and using
	$\sum_Qx_Q^D\le(\sum_Qx_Q)^D$ for $D\ge1$.
\end{proof}

\begin{proposition}\label{prop:case1-moment-tests}
	Assume Case~1.  There exist positive numbers $\eta_u$ and, when $p>1$,
	$\eta_v$, depending only on the exponents, such that
	\begin{equation}\label{eq:case1-forward-moment}
		\zeta
		\lesssim
		A\,\mathfrak C_{u,\eta_u}(\Sfam)^{1/q},
	\end{equation}
	and, when $p>1$,
	\begin{equation}\label{eq:case1-dual-moment}
		\zeta^*
		\lesssim
		A\,\mathfrak C_{v,\eta_v}(\Sfam)^{1/p'}.
	\end{equation}
\end{proposition}

\begin{proof}
	Assume first that $\Sfam$ is finite.  For the forward estimate, choose
	$d>1$ sufficiently close to one that
	\begin{equation}\label{eq:case1-d-forward}
		\frac1a-d\theta_u\ge0,
		\qquad
		\frac1{b'}-d\theta_v\ge0,
		\qquad
		d<q',
		\qquad
		(d-1)(\gamma-\sgap)\le\sgap.
	\end{equation}
	Such a $d$ exists: if $p<q$, then $\sgap>0$, while if $p=q$, Case~1 forces
	$\gamma=\sgap=0$.
	
	For $Q\in\Sfam$, set
	\[
	I_Q:=\sum_{\substack{P\in\Sfam\\P\subseteq Q}}\lambda_P.
	\]
	Since $K_P\lesssim_X A$, one has
	\begin{align*}
		\lambda_P
		&\lesssim_X
		A^d\mu(P)^{\kappa(d-1)}
		u(P)^{1/a-d\theta_u}
		v(P)^{1/b'-d\theta_v}.
	\end{align*}
	The three exponents are nonnegative, $\kappa(d-1)>0$, and
	\begin{align*}
		&\kappa(d-1)
		+\left(\frac1a-d\theta_u\right)
		+\left(\frac1{b'}-d\theta_v\right)\\
		&\qquad=1+\sgap-(d-1)(\gamma-\sgap)\ge1.
	\end{align*}
	Thus \cref{lem:sparse-monomial} gives
	\begin{equation}\label{eq:case1-descendant}
		I_Q
		\lesssim
		A^d\mu(Q)^{\kappa(d-1)}
		u(Q)^{1/a-d\theta_u}
		v(Q)^{1/b'-d\theta_v}.
	\end{equation}
	
	Let
	\[
	c_Q:=\avg{u}_Q^{1/a}\avg{v}_Q^{-1/b}\mu(Q)^\gamma,
	\qquad
	c_Qv(Q)=\lambda_Q.
	\]
	Apply \cref{lem:laminar-operator} with $\nu=v$ and $t=q$.  Put
	\[
	\lambda_q:=q-d(q-1)>0.
	\]
	Inserting \eqref{eq:case1-descendant} and collecting the powers of
	$\mu(Q),u(Q),v(Q)$ gives
	\[
	\norm{T_Ru}{L^q(v)}^q
	\lesssim
	A^{d(q-1)}
	\sum_{Q\subseteq R}K_Q^{\lambda_q}u(Q)^{q/p}.
	\]
	By \eqref{eq:defect-factor-upper} and
	$d(q-1)+\lambda_q=q$,
	\begin{align*}
		\norm{T_Ru}{L^q(v)}^q
		&\lesssim
		A^q
		\sum_{Q\subseteq R}
		(\mathfrak d_Q^u)^{\lambda_q\theta_u}
		(\mathfrak d_Q^v)^{\lambda_q\theta_v}
		u(Q)^{q/p}\\
		&\le
		A^qu(R)^{q/p-1}
		\sum_{Q\subseteq R}
		(\mathfrak d_Q^u)^{\lambda_q\theta_u}u(Q).
	\end{align*}
	Therefore \eqref{eq:case1-forward-moment} holds with
	$\eta_u:=\lambda_q\theta_u>0$.
	
	When $p>1$, choose $d>1$ satisfying the first, second, and fourth
	conditions in \eqref{eq:case1-d-forward}, together with $d<p$.  Apply the
	laminar estimate with $\nu=u$ and $t=p'$.  If
	\[
	\lambda_{p'}:=p'-d(p'-1)>0,
	\]
	the same exponent calculation yields
	\[
	\norm{T_Rv}{L^{p'}(u)}^{p'}
	\lesssim
	A^{p'}v(R)^{p'/q'-1}
	\sum_{Q\subseteq R}
	(\mathfrak d_Q^v)^{\lambda_{p'}\theta_v}v(Q).
	\]
	This proves \eqref{eq:case1-dual-moment} with
	$\eta_v:=\lambda_{p'}\theta_v>0$.
\end{proof}

\begin{proof}[Proof of \cref{thm:weakA-bilinear-intro}, Case~1]
	For finite $\Sfam$, combine
	\cref{prop:case1-moment-tests,cor:defect-moment-packing} to obtain
	\[
	\zeta\lesssim A U^{1/q},
	\qquad
	\zeta^*\lesssim A V^{1/p'}\quad(p>1).
	\]
	Applying \cref{lem:bilinear-testing} proves the estimate for finite
	$\Sfam$.  The countable case follows from
	\cref{lem:finite-truncations}.
\end{proof}

\subsection{Case~2}

\cref{lem:sparse-monomial} requires $\kappa>0$ and the possibility
of choosing $d>1$ in \eqref{eq:case1-d-forward}.  These conditions fail at the upper endpoint and on the positive diagonal.  In this regime we instead keep the defect powers of both weights and apply H\"older's
inequality directly to the descendant sum.

\begin{lemma}
	\label{lem:case2-descendant}
	Assume Case~2 and let $u,v\in\weakA$.  For every
	$s\in[p,q]\cap(1,\infty)$ and every $Q\in\D$,
	\begin{equation}\label{eq:case2-descendant}
		\sum_{\substack{P\in\Sfam\\P\subseteq Q}}\lambda_P
		\lesssim
		A\,U^{1/s}V^{1/s'}
		u(Q)^{1/p}v(Q)^{1/q'}.
	\end{equation}
	The implicit constant is uniform for
	$s\in[p,q]\cap(1,\infty)$.
\end{lemma}

\begin{proof}
	By \eqref{eq:defect-factor-upper} and H\"older's inequality,
	\begin{align*}
		\sum_{P\subseteq Q}\lambda_P
		&\lesssim_X A
		\left(
		\sum_{P\subseteq Q}
		(\mathfrak d_P^u)^{s\theta_u}u(P)^{s/p}
		\right)^{1/s}\\
		&\qquad\times
		\left(
		\sum_{P\subseteq Q}
		(\mathfrak d_P^v)^{s'\theta_v}v(P)^{s'/q'}
		\right)^{1/s'}.
	\end{align*}
	Since $p\le s\le q$, both $s/p$ and $s'/q'$ are at least one.  Hence
	\begin{align*}
		\sum_{P\subseteq Q}
		(\mathfrak d_P^u)^{s\theta_u}u(P)^{s/p}
		&\le
		u(Q)^{s/p-1}
		\sum_{P\subseteq Q}
		(\mathfrak d_P^u)^{s\theta_u}u(P),\\
		\sum_{P\subseteq Q}
		(\mathfrak d_P^v)^{s'\theta_v}v(P)^{s'/q'}
		&\le
		v(Q)^{s'/q'-1}
		\sum_{P\subseteq Q}
		(\mathfrak d_P^v)^{s'\theta_v}v(P).
	\end{align*}
	Apply \cref{cor:defect-moment-packing} to the two sums.  The factors
	$1+(s\theta_u)^{-1}$ and $1+(s'\theta_v)^{-1}$ are uniformly bounded in the
	admissible interval, which proves \eqref{eq:case2-descendant}.
\end{proof}

\begin{proposition}\label{prop:case2-tests}
	Under the assumptions of \cref{lem:case2-descendant}, for every admissible
	$s$,
	\begin{align}
		\zeta
		&\lesssim
		A\,U^{1/q}
		\bigl(U^{1/s}V^{1/s'}\bigr)^{1/q'},
		\label{eq:case2-forward-s}\\
		\zeta^*
		&\lesssim
		A\,V^{1/p'}
		\bigl(U^{1/s}V^{1/s'}\bigr)^{1/p},
		\qquad p>1.
		\label{eq:case2-dual-s}
	\end{align}
\end{proposition}

\begin{proof}
	Apply \cref{lem:laminar-operator} to $T_Ru$ with $\nu=v$ and $t=q$.
	Using \eqref{eq:case2-descendant} gives
	\begin{align*}
		\norm{T_Ru}{L^q(v)}^q
		&\lesssim
		A^q\bigl(U^{1/s}V^{1/s'}\bigr)^{q-1}\\
		&\qquad\times
		\sum_{Q\subseteq R}
		(\mathfrak d_Q^u)^{\theta_u}
		(\mathfrak d_Q^v)^{\theta_v}u(Q)^{q/p}.
	\end{align*}
	Since $q/p\ge1$ and $\mathfrak d_Q^v\le1$,
	\[
	\sum_{Q\subseteq R}
	(\mathfrak d_Q^u)^{\theta_u}
	(\mathfrak d_Q^v)^{\theta_v}u(Q)^{q/p}
	\le
	u(R)^{q/p-1}
	\sum_{Q\subseteq R}(\mathfrak d_Q^u)^{\theta_u}u(Q).
	\]
	The $u$-moment estimate proves \eqref{eq:case2-forward-s}.  The dual
	estimate follows in the same way from the laminar potential estimate with
	$\nu=u$ and $t=p'$, using the $v$-moment in the outer sum.
\end{proof}

\begin{lemma}\label{lem:mixed-optimization}
	For $U,V>0$,
	\[
	\inf_{s\in[p,q]\cap(1,\infty)}U^{1/s}V^{1/s'}
	=\mathfrak m_{p,q}(U,V).
	\]
	When $p=1$, the value at $s=1$ is interpreted as the limit $s\downarrow1$.
\end{lemma}

\begin{proof}
	Since
	\[
	\log\bigl(U^{1/s}V^{1/s'}\bigr)
	=\log V+\frac1s\log\frac UV,
	\]
	the expression is affine in $1/s$.  Its minimum over the admissible
	interval is therefore attained at an endpoint, with the limiting
	interpretation when $p=1$.
\end{proof}

\begin{proof}[Proof of \cref{thm:weakA-bilinear-intro}, Case~2]
	For finite $\Sfam$, optimize \eqref{eq:case2-forward-s} and, when
	$p>1$, \eqref{eq:case2-dual-s} by \cref{lem:mixed-optimization}, and
	then apply \cref{lem:bilinear-testing}.  The countable case follows from
	\cref{lem:finite-truncations}.
\end{proof}

\section{Applications}

We apply the preceding estimates to
\[
\mathcal A_{\Sfam}^{r,\xi}F(x)
:=\left[
\sum_{Q\in\Sfam}
\left(\mu(Q)^{\xi-1}\int_Q |F|\,\dd\mu\right)^r
\one_Q(x)
\right]^{1/r},
\qquad 0<r<\infty.
\]
This agrees with the operator $\mathcal A_{\Sfam}^{r,\alpha}$ of \citet{FH}
after the reparameterization $\alpha=1-\xi$.
Put $\delta_{p,q}=1/p-1/q$.
Their Euclidean result covers
\[
1<p\le q<\infty,
\qquad
\delta_{p,q}\le\xi<1,
\]
not only the Sobolev line $\xi=\delta_{p,q}$.
For a Euclidean sparse family, let
\[
[w,\sigma]_{\mathsf A_{p,q}^{\xi}(\Sfam)}
:=\sup_{Q\in\Sfam}
|Q|^{\xi-1}w(Q)^{1/q}\sigma(Q)^{1/p'}.
\]
Write $\mathsf E$ for the exceptional range $p=q>r$ and $\xi>0$, and set $\vartheta=r/p$.
If $w,\sigma\in A_\infty(\mathbb R^n)$, then \citet{FH} prove that boundedness
is equivalent to the finiteness of this characteristic and that
\[
\norm{\mathcal A_{\Sfam}^{r,\xi}(\sigma\cdot)}
{L^p(\sigma)\to L^q(w)}
\lesssim
[w,\sigma]_{\mathsf A_{p,q}^{\xi}(\Sfam)}
\Phi_{A_\infty}(w,\sigma),
\]
where
\[
\Phi_{A_\infty}(w,\sigma)
:=
\begin{dcases}
	[\sigma]_{A_\infty}^{1/q}
	+[w]_{A_\infty}^{(1/r-1/p)_+},
	& \neg\mathsf E,\\[2mm]
	\begin{aligned}
		&[w]_{A_\infty}^{(1-\vartheta)^2/r}
		[\sigma]_{A_\infty}^{(1-(1-\vartheta)^2)/r}\\[-1mm]
		&\quad+
		[w]_{A_\infty}^{(1-\vartheta^2)/r}
		[\sigma]_{A_\infty}^{\vartheta^2/r},
	\end{aligned}
	& \mathsf E.
\end{dcases}
\]
Thus the larger range requires no new pointwise model; only the diagonal
positive-fractional case produces a different mixed factor.

On a space of homogeneous type, define
\[
\Achar_{p,q}^{\xi}(w,\sigma;\Sfam)
:=\sup_{Q\in\Sfam}
\mu(\wideB{Q})^{\xi-1}
w(\wideB{Q})^{1/q}
\sigma(\wideB{Q})^{1/p'}.
\]
It is dominated by the global ball characteristic
\[
\Achar_{p,q}^{\xi}(w,\sigma)
:=\sup_B
\mu(B)^{\xi-1}w(B)^{1/q}\sigma(B)^{1/p'}.
\]

\begin{corollary}[Weak-$A_\infty$ estimate beyond the Sobolev line]
	\label{cor:vector-weakA-general-xi}
	Let $1<p\le q<\infty$, $0<r<\infty$, and $\delta_{p,q}\le\xi<1$.
	Let $\Sfam\subset\D$ be sparse, assume $\sigma\in\weakA$, and assume also $w\in\weakA$ when $r<p$.
	Then
	\[
	\norm{\mathcal A_{\Sfam}^{r,\xi}(\sigma\cdot)}
	{L^p(\sigma)\to L^q(w)}
	\lesssim
	\Achar_{p,q}^{\xi}(w,\sigma;\Sfam)
	\Phi_{\weakA}(w,\sigma),
	\]
	where
	\[
	\Phi_{\weakA}(w,\sigma)
	:=
	\begin{dcases}
		[\sigma]_{\weakA}^{1/q}
		+\one_{\{r<p\}}[w]_{\weakA}^{1/r-1/p},
		& \neg\mathsf E,\\[2mm]
		\begin{aligned}
			&[w]_{\weakA}^{(1-\vartheta)^2/r}
			[\sigma]_{\weakA}^{(1-(1-\vartheta)^2)/r}\\[-1mm]
			&\quad+
			[w]_{\weakA}^{(1-\vartheta^2)/r}
			[\sigma]_{\weakA}^{\vartheta^2/r},
		\end{aligned}
		& \mathsf E.
	\end{dcases}
	\]
	The local characteristic may be replaced by $\Achar_{p,q}^{\xi}(w,\sigma)$.
	The implicit constant depends only on $p,q,r,\xi$ and the structural constants of $X$.
\end{corollary}

\begin{proof}
	It suffices first to take $\Sfam$ finite.
	For $0<r\le p$, apart from the corner $r=p=q$, put
	\[
	P=\frac{p}{r},
	\qquad
	Q_0=\frac{q}{r},
	\qquad
	\gamma=r\xi.
	\]
	For nonnegative $f,h$,
	\[
	\Lform_{\Sfam}^{1/r,\infty,r\xi}
	\bigl((f\sigma)^r,hw\bigr)
	=
	\int_X
	\bigl[\mathcal A_{\Sfam}^{r,\xi}(f\sigma)\bigr]^r
	h\,\dd w,
	\]
	so duality gives
	\[
	\norm{\mathcal A_{\Sfam}^{r,\xi}(\sigma\cdot)}
	{L^p(\sigma)\to L^q(w)}^r
	=
	\Nform_{\Sfam}^{1/r,\infty,r\xi;P,Q_0}
	\bigl(\sigma^{1-p},w^{1-Q_0'}\bigr).
	\]
	The transformed weights and parameters are
	\[
	u=\sigma,
	\qquad
	v=w,
	\qquad
	\theta_u=\frac{r}{p'},
	\qquad
	\theta_v=\frac{r}{q},
	\qquad
	\kappa=r(1-\xi),
	\]
	and
	\[
	\Achar_{\Sfam}^{r\xi}(u,v)
	=\Achar_{p,q}^{\xi}(w,\sigma;\Sfam)^r.
	\]
	Since $r(1/p-1/q)\le r\xi<r$, \cref{thm:weakA-bilinear-intro} applies.
	Case~1 yields the first branch of $\Phi_{\weakA}$, including the endpoint
	$r=p<q$; Case~2 occurs exactly in $\mathsf E$ and gives the second branch
	after taking the $r$th root.
	
	If $r=p=q$, defect factorization and \cref{cor:defect-moment-packing} imply
	\[
	\norm{\mathcal A_{\Sfam}^{p,\xi}(f\sigma)}{L^p(w)}^p
	\lesssim
	\Achar_{p,p}^{\xi}(w,\sigma;\Sfam)^p
	\sum_{Q\in\Sfam}
	(\mathfrak d_Q^\sigma)^{p-1}
	\avg{|f|}_{Q,\sigma}^p
	\sigma(Q).
	\]
	The weighted Carleson embedding theorem bounds the sum by $[\sigma]_{\weakA}\norm{f}{L^p(\sigma)}^p$.
	If $r>p$, pointwise monotonicity gives
	$\mathcal A_{\Sfam}^{r,\xi}F\le\mathcal A_{\Sfam}^{p,\xi}F$.
	Approximation by finite subfamilies completes the proof.
\end{proof}

The same estimate holds for every operator admitting pointwise sparse domination
by finitely many models $\mathcal A_{\Sfam}^{r,\xi}$, with the implicit
constant also depending on the number of models and the uniform
sparse-domination constant.


\end{document}